\documentclass[11pt]{article}     

\usepackage{graphicx}
\usepackage{url}
\usepackage{color}
\definecolor{maroon}{rgb}{.69,.188,.376}
\definecolor{darkgreen}{rgb}{0,.5,0}
\definecolor{darkblue}{rgb}{0,0,.5}
\definecolor{magenta}{rgb}{1,0,1}

\usepackage[mathscr]{euscript}		

\usepackage{dsfont}

\usepackage{psfrag}			
\usepackage[colorlinks=true]{hyperref}
\hypersetup{pdftex, colorlinks=true, linkcolor=maroon, citecolor=maroon,
  filecolor=blue,urlcolor=blue}

\usepackage{amsmath}
\usepackage{amsthm}

\usepackage{amssymb}

\usepackage{caption} 

\usepackage{anysize}

\usepackage{enumerate}
\usepackage{enumitem}

\usepackage[top=.95in, bottom = 1 in, left=1in, right = .6in]{geometry}

\newcommand{\ind}{\mathds}
\newcommand{\floor}[1]{{\lfloor #1 \rfloor}}
\newcommand{\Z}{\ensuremath{\mathbb{Z}}}
\newcommand{\N}{\ensuremath{\mathbb{N}}}
\newcommand{\R}{\ensuremath{\mathbb{R}}}

\newcommand{\asto}[1]{\underset{{#1}\to\infty}{\longrightarrow}}

\newcommand{\E}{\ensuremath{\mathbb{E}}}
\renewcommand{\P}{\ensuremath{\mathbb{P}}}

\newtheorem{theorem}{Theorem}[section]
\newtheorem{lemma}[theorem]{Lemma}

\newtheorem{proposition}[theorem]{Proposition}
\newtheorem{remark}[]{Remark}

\numberwithin{equation}{section}

\definecolor{Red}{rgb}{1,0,0}
\definecolor{Blue}{rgb}{0,0,1}
\definecolor{Olive}{rgb}{0.41,0.55,0.13}
\definecolor{Yarok}{rgb}{0,0.5,0}
\definecolor{Green}{rgb}{0,1,0}
\definecolor{MGreen}{rgb}{0,0.8,0}
\definecolor{DGreen}{rgb}{0,0.55,0}
\definecolor{Yellow}{rgb}{1,1,0}
\definecolor{Cyan}{rgb}{0,1,1}
\definecolor{Magenta}{rgb}{1,0,1}
\definecolor{Orange}{rgb}{1,.5,0}
\definecolor{Violet}{rgb}{.5,0,.5}
\definecolor{Purple}{rgb}{.75,0,.25}
\definecolor{Brown}{rgb}{.75,.5,.25}
\definecolor{Grey}{rgb}{.7,.7,.7}
\definecolor{Black}{rgb}{0,0,0}

\newcommand{\ignore}[1]{{}}

\date{\today}
\begin{document}
\baselineskip=14pt

\title{Subdiffusivity of a random walk among a Poisson system of moving traps on $\Z$}

\author{
Siva Athreya%
  \thanks{8th Mile Mysore Road, Indian Statistical Institute,
         Bangalore 560059, India.
    Email: \url{athreya@ms.isibang.ac.in}}
  \and
  Alexander Drewitz%
  \thanks{Universit\"at zu K\"oln,
Mathematisches Institut,
Weyertal 86--90,
50931 K\"oln, Germany.
    Email: \url{drewitz@math.uni-koeln.de}}
  \and
 Rongfeng Sun%
  \thanks{Department of Mathematics, National University of Singapore,
S17, 10 Lower Kent Ridge Road
Singapore, 119076.
    Email: \url{matsr@nus.edu.sg}}
}

\maketitle

\begin{abstract}
We consider a random walk among a Poisson system of moving traps on
$\Z$. In earlier work \cite{DrGaRaSu-10}, the quenched and annealed
survival probabilities of this random walk have been
investigated. Here we study the path of the random walk conditioned on
survival up to time $t$ in the annealed case and show that it is
subdiffusive. As a by-product, we obtain an upper bound on the number
of so-called {\em thin points} of a one-dimensional random walk, as
well as a bound on the total volume of the holes in the random walk's
range.
\end{abstract}


\section{Introduction}

Trapping problems have been studied in the statistical physics and
probability literature for decades, where a particle modeled by a
random walk or Brownian motion is killed when it meets one of the
traps. When the traps are Poisson distributed in space and immobile,
much has been understood, see e.g.\ the seminal works of Donsker and
Varadhan \cite{DoVa-75,DoVa-79} as well as the monograph by Sznitman
\cite{Sz-98} and the references therein. However, when the traps are
mobile, surprisingly little is known. In a previous work
\cite{DrGaRaSu-10} (see also \cite{PSSS13}), the long-time asymptotics
of the annealed and quenched survival probabilities were identified in
all dimensions, extending earlier work in the physics
literature~\cite{MOBC03, MOBC04}. The goal of the current work is to
investigate the path behavior of the one-dimensional random walk
conditioned on survival up to time $t$ in the annealed setting, which
is the first result of this type to our best knowledge. Note that the
model of random walk among mobile traps is a natural model for many
physical and biological phenomena, such as foraging predators vs prey,
or diffusing T-cells vs cancer cells in the blood stream.

We now recall the model considered in \cite{DrGaRaSu-10}. Given an
intensity parameter $\nu > 0$, we consider a family of i.i.d.\ Poisson
random variables $(N_y)_{y \in \Z^d}$ with mean $\nu$. Given
$(N_y)_{y\in\Z^d}$, we then start a family of independent simple
symmetric random walks $(Y^{j,y})_{y \in \Z^d, \; 1 \le j \le N_y}$ on
$\Z^d$, each with jump rate $\rho \ge 0$, with
$Y^{j,y}:=(Y^{j,y}_t)_{t\geq 0}$ representing the path of the $j$-th
trap starting from $y$ at time $0$. We will refer to these as
\lq$Y$-particles\rq\ or \lq traps\rq. For $t \ge 0$ and $x \in \Z^d,$
we denote by
\begin{equation}
\xi(t,x) := \sum_{y \in \Z^d, \; 1 \le j \le N_y} \delta_x (Y^{j,y}_t) \label{xidef}
\end{equation}
the number of traps at site $x$ at time $t.$

Let $X:=(X_t)_{t\geq 0}$ denote a simple symmetric random walk on
$\Z^d$ with jump rate $\kappa \ge 0$ (and later on a more general
random walk, see Theorem \ref{thm:main}) that evolves independently of
the $Y$-particles.  At each time $t$, the $X$ particle is killed with
rate $\gamma \xi(t,X_t),$ where $\gamma \ge 0$ is the interaction
parameter -- i.e., the killing rate is proportional to the number of
traps that the $X$ particle sees at that time instant. We denote the
probability measure underlying the $X$ and $Y$ particles by $\P$, and
if we consider expectations or probabilities with respect to only a
subset of the defined random variables, we give those as a
superscript, and sometimes also specify the starting configuration as
a subscript, such as $\P_0^X.$

Conditional on the realization of $\xi,$ the survival probability of
$X$ up to time $t$ is then given by
\begin{equation} \label{eq:quenchedSP}
 \E_0^X \Big[ \exp \Big \{ -\gamma \int_0^t \xi(s,X_s) \, {\rm d}s \Big\} \Big].
\end{equation}
This quantity is also referred to as the \lq quenched survival probability\rq. Taking expectation with respect to $\xi$ yields the \lq annealed  survival probability\rq
$$
Z^\gamma_t:= \E^\xi \Big[  \E_0^X \Big[ \exp \Big \{ -\gamma \int_0^t \xi(s,X_s) \, {\rm d}s \Big\} \Big]\Big].
$$

Since we will mainly be interested in the behavior of $X$, it is useful to integrate out $\xi$ in order to obtain
the annealed survival probability for a given realization of $X,$ i.e.,
\begin{equation} \label{eq:condAnnProb}
Z^\gamma_{t,X} := \E^\xi \Big[ \exp \Big\{ - \gamma \int_0^t \xi(s,X_s) \, {\rm d}s \Big\}
\Big].
\end{equation}

Note that the annealed survival probability $Z^\gamma_t$ is also given by $\E^X_0 [Z^\gamma_{t,X}].$
In \cite{DrGaRaSu-10}, the following asymptotics for the annealed survival probability have been derived.

\begin{theorem}{\rm\cite[Thm. 1.1]{DrGaRaSu-10}} \label{thm:annealedAsymptotics}
Assume that $\gamma\in (0,\infty]$, $\kappa\geq 0$, $\rho>0$ and $\nu>0$, then
\begin{align*}
  \E^X_0[Z^\gamma_{t,X}] = \left\{
 \begin{array}{ll}
\exp\Big\{-\nu \sqrt{\frac{8\rho t}{\pi}}(1+o(1))\Big\}, &  d=1, \\
&\\
\exp\Big\{-\nu\pi\rho \frac{t}{\ln t}(1+o(1))\Big\}, & d=2,  \\
&\\
\exp\Big\{-\lambda_{d,\gamma,\kappa,\rho, \nu}\, t(1+o(1))\Big\}, & d\geq 3,
\end{array}
\right.
\end{align*}
where $\lambda_{d,\gamma,\kappa,\rho, \nu}$ depends on $d$, $\gamma$, $\kappa$, $\rho$, $\nu$, and is called the annealed Lyapunov exponent. 
\end{theorem}

\begin{remark}{\rm
The annealed and quenched survival probabilities introduced above are closely related to  the parabolic Anderson model, namely, the solution of the lattice stochastic heat equation with a random potential $\xi$:
\begin{align*}
\left\{ \begin{array}{rll}
\frac{\partial}{\partial t} u(t, x) &= \kappa \Delta u(t, x) - \gamma\, \xi(t, x)\, u(t, x),
&(t,x) \in [0,\infty) \times \Z,\\
u(0,x) \hspace{-,5em}&= 1, & x \in \Z.
\end{array} \right.
\end{align*}
See \cite{DrGaRaSu-10} for more details.
}
\end{remark}

It is natural to ask how the asymptotics in Theorem
\ref{thm:annealedAsymptotics} are actually achieved, both in terms of
the behavior of $X$ as well as that of $\xi.$ We consider the case
$d=1$ and investigate the typical behavior of $X$ conditioned on
survival. In the next section we state the model precisely and the
main results of the paper.

\subsection{Main Results} \label{mr}

 We shall consider the model considered in \cite{DrGaRaSu-10} but will
 allow the following generalisations:
\begin{eqnarray}
&&\mbox{$X$ is a continuous time
 random walk on $\Z$ with jump rate $\kappa>0$, and possess a  jump kernel} \nonumber \\ && \mbox{ $p_X$ which is non-degenerate with zero mean.} \label{grwx}\\
&& \nonumber\\
&&\mbox{$Y$-particles (traps) are independent continuous time random walks on $\Z$ with jump rate } \nonumber \\ &&\mbox{$\rho>0$, whose jump kernel $p_Y$ is symmetric. } \label{grwy}
\end{eqnarray}
As defined earlier, $\xi$ is as in (\ref{xidef}) and we shall assume
that the interaction parameter $\gamma\in (0,\infty]$ and the trap
  intensity $\nu>0$.

Before stating our results, we introduce some notation. For $t \in
(0,\infty)$ and a c\`adl\`ag function $f \in D([0,t], \R)$ (with
$D([0,t], \R)$ denoting the Skorokhod space), we define its supremum
norm by
\begin{equation}
\Vert f \Vert_t := \sup_{x \in [0,t]} \vert f(x) \vert.
\end{equation}

\subsubsection{Sub-diffusivity of $X$}

We are interested in the (non-consistent) family of Gibbs measures
\begin{equation} \label{eq:Gibbs}
P_{t}^\gamma ( X\in \cdot ) := \frac{\E_0^X \Big[ \E^\xi  \Big[ \exp \Big \{ -\gamma \int_0^t \xi(s,X_s) \, {\rm d}s \Big\} \Big] \ind{1}_{X \in \cdot} \Big] }
{  \E^X_0[Z^\gamma_{t,X}]}, \quad t \ge 0,
\end{equation}
on the space of c\`adl\`ag paths on $Z$. We will bound typical
fluctuations of $X$ with respect to $P_{t}^{\gamma}$. Our primary result
is the following bound on the fluctuation of $X$ conditioned on
survival up to time $t$.

\begin{theorem} \label{thm:main}
Let $X$ and $Y$ be as  (\ref{grwx}) and (\ref{grwy}) respectively. Assume that $\exists \lambda_* >0 $ such that 
\begin{equation} \label{expmom}\sum_{x\in\Z} e^{\lambda_* |x|} p_X(x)<\infty \mbox{ and
}\sum_{x\in\Z} e^{\lambda_* |x|} p_Y(x)<\infty.
\end{equation}  Then there exists $\alpha>0$ such that for all
  $\epsilon>0$,
\begin{equation}\label{eq:flucbd}
P_{t}^{\gamma} \Big( \Vert X \Vert_t \in \big(\alpha t^{\frac13} ,\, t^{\frac{11}{24}+\epsilon}\big) \Big) \asto{t} 1.
\end{equation}
\end{theorem}

\begin{remark}{\rm Since $\frac{11}{24}<\frac{1}{2}$, the above result shows that $X$ is
sub-diffusive under $P^\gamma_t$. We believe that $X$ in fact
fluctuates on the scale of $t^{1/3}$ (modulo lower order
corrections). Interestingly, this would coincide with the fluctuation
known for the case of immobile traps in dimension one (see
e.g.\ \cite{S90, S03}), and we conjecture that even the rescaled path
converges to the same limit. We also note that this should happen even
though the annealed survival probability decays at a different rate
when the traps are mobile. However, the mobile trap case presents
fundamental difficulties that are not present in the immobile case.  }
\end{remark}

\subsubsection{Thin points of $X$}

As by-product of our analysis, we obtain bounds on the number of
thin-points of a one-dimensional random walk which is of independent
interest. Let $X$ be as in (\ref{grwx}) and
\begin{equation}\label{loct}
L_t(x) := L_t^{X}(x) := \int_0^t \delta_{x}(X_s) \, {\rm d}s
\end{equation}
denote the local time of the random walk $X$ at $x$ up to time $t$.
Typically, for $x\in \Z$ in the bulk of the range of $X$, the local
time $L_t^X(x)$ will be of order $\sqrt t.$ We are interested in {\em
  thin points}. More precisely, for $M > 0$, a point $x$ in the range
of $X$ is called \lq$M$-thin at time $t$\rq\ if $L_t(x) \in(0, M]$,
  and we denote by
\begin{equation}
\mathcal T_{t, M}  := \{ x \in \Z \, : \, L_t(x) \in (0, M]\}
\end{equation}
the set of $M$-thin points at time $t$.

For $\gamma > 0$, we introduce the local time functional
\begin{equation} \label{eq:localTimeFunc}
F_t^{\gamma}(X) := \sum_{x \in \Z} e^{-\gamma L_t^{X}(x)}
\ind{1}_{L_t^X(x) > 0}.
\end{equation}
For $X$ with mean zero and finite variance for its increments, Proposition \ref{prop:expMomLocal} below implies the existence of constants $c(\gamma), C(\gamma) \in (0,\infty)$ such that for all $t \in (0,\infty),$
\begin{equation}
\E_0^X \Big[ \exp \Big\{ \frac{c(\gamma)}{1 \vee \ln t} F_t^{\gamma}(X)\Big\} \Big] \le C(\gamma).
\end{equation}
Since $F_t^\gamma(X)\geq e^{-\gamma M} \vert \mathcal T_{t, M} \vert$, we immediately obtain the following result
\begin{theorem} \label{thm:thinPoints}
Let $\gamma \in (0,\infty)$, and let $X$ be as in (\ref{grwx}). Assume that 
$$\sum_{x\in\Z} x^2 p_X(x)<\infty. $$ Then, for any positive $M$,
\begin{equation}
\P \big ( \vert \mathcal T_{t, M}\vert \geq a) \leq C(\gamma) e^{-\frac{c(\gamma)e^{-\gamma M}}{1\vee \ln t} a} \qquad \mbox{for all } a>0.
\end{equation}
\end{theorem}
\begin{remark}{\rm
Thin points of Brownian motion in dimension $d\geq 2$ have been
studied in \cite{DePeRoZe-00} using L\'evy's modulus of
continuity. Dimension $1$ is different and could be analyzed by using
the Ray-Knight theorem. When $X$ is a simple random walk on $\Z$,
there is still a Ray-Knight theorem to aid our analysis. But for a
general random walk $X$ as in Theorem \ref{thm:thinPoints}, this
approach fails.  }
\end{remark}

\subsubsection{Holes in the range of $X$}
For a simple random walk $X$ on $\Z$, its range equals the interval $[\inf_{0\leq s\leq t}X_s, \sup_{0\leq s\leq t} X_s]$, which is no longer true for non-simple random walks. However, for  $X$ as in (\ref{grwx}), we can control the difference
\begin{equation}\label{GtX}
G_t(X) := \big(\sup_{0\leq s\leq t} X_s - \inf_{0\leq s\leq t} X_s\big) - |{\rm Range}_{s\in [0,t]}(X_s)|=
\sum_{\inf_{s\in [0,t]}X_s <x< \sup_{s\in [0,t]} X_s} \ind{1}_{L^X_t(x)=0}.
\end{equation}
This is the total volume of the holes in the range of $X$ by time $t$, which will appear in the proof of Theorem \ref{thm:main} for non-simple random walks.

\begin{theorem} \label{thm:range}
Let $X$ be as in (\ref{grwx}). Assume that $\exists \lambda_* >0 $ such that 
\begin{equation} \label{expmomy} \sum_{x\in\Z} e^{\lambda_* |x|} p_X(x)<\infty.\end{equation}Then there exist $c, C>0$ such that for $\lambda_t := \frac{c}{1 \vee \ln t}$, we have
\begin{equation}\label{1.12}
\E_0^X \big[ \exp \{ \lambda_t G_t(X)\} \big] \le C \qquad \mbox{for all } t \in (0,\infty).
\end{equation}
\end{theorem}

As a consequence of \eqref{1.12}, we have
\begin{equation}\label{EGt}
\E^X_0[G_t(X)] = \int_0^\infty \P^X_0(G_t(X) \geq m)\, {\rm d}m \le \int_0^\infty C e^{-\frac{m c}{1\vee \ln t}} {\rm d} m \le C\ln t.
\end{equation}
\begin{remark}{\rm
We note that \eqref{1.12} cannot hold if $\sum_{|x|>L}p_X(x)$ has power law decay. This is easily seen by considering the strategy that the random walk makes a single jump from $0$ to a position $x\geq t$, and then never falls below $x$ before time $t$. The probability of this strategy decays polynomially in $t$, while the gain $e^{\lambda_t G_t(X)}$ is more than the stretched exponential.
}
\end{remark}

Throughout the paper, $c$ and $C$ will denote generic constants,
whose values may change from line to line. Indexed constants such as
$c_1$ and $C_2$ will denote values that will be fixed from their first
occurrence onwards. In order to emphasise dependence of a constant on
a parameter, we will write $C(p)$ for instance.

{\bf Layout :} The rest of the paper is organised as follows. In Section \ref{pthin} we prove Theorem \ref{thm:thinPoints}, in Section
\ref{prange} we prove Theorem \ref{thm:range}. We conclude the paper
with Section \ref{pmain} where  we prove Theorem \ref{thm:main}.

\section{Proof of Theorem \ref{thm:thinPoints}} \label{pthin}
Recall from \eqref{eq:localTimeFunc} that $F_t^{\gamma}(X) := \sum_{x \in \Z} e^{-\gamma L_t^{X}(x)} \ind{1}_{L_t^X(x) > 0}$.
As remarked before the statement of Theorem \ref{thm:thinPoints}, it suffices to establish the following result
\begin{proposition} \label{prop:expMomLocal}
Let $X$ be a random walk satisfying the assumptions of Theorem \ref{thm:thinPoints}. Then for each $\gamma > 0$, there exist constants $c(\gamma), C(\gamma)\in (0,\infty)$, such that for $\lambda_t := \frac{c(\gamma)}{1 \vee \ln t}$, we have
\begin{equation}\label{prop3.1}
\E_0^X \big[ \exp \{ \lambda_t F_t^{\gamma}(X)\} \big] \le C(\gamma) \qquad \mbox{for all } t \in (0,\infty).
\end{equation}
\end{proposition}
We will prove Proposition \ref{prop:expMomLocal} by approximating $X$ by a sequence of discrete time random walks. More precisely, for any $0<q<\frac1\kappa$, where $\kappa$ is the jump rate of $X$, let $X^q$ denote the discrete time random walk with transition probability
\begin{equation}\label{32}
\P_0^{X^q}(X^q(1) = 0) = 1-\kappa q, \quad \text{ and }
\quad \P_0^{X^q}(X^q(1) = x)
 = \kappa q p_X( x) \quad
\forall x \in \Z,
\end{equation}
where $p_X(\cdot)$ is the jump probability kernel of $X$. Let $X^q(s) :=X^q(\floor{s})$ for all $s\geq 0$. It is then a standard fact that the sequence of discrete time random walks $(X^q(s/q))_{s\geq 0}$ converges in distribution to $(X_s)_{s\geq0}$ as $q\downarrow 0$. Proposition \ref{prop:expMomLocal} will then follow from its analogue for $(X^q(s/q))_{s\geq 0}$, together with the following lemma.

\begin{lemma} \label{lem:discreteTimeApprox}
Let $X$ be an arbitrary continuous time random walk on $\Z$, and let $X^q(\cdot/ q)$ be its discrete time approximation defined above. Then for any $\lambda, t \in [0,\infty),$
\begin{equation}\label{discapprox}
\lim_{n \to \infty} \E_0^{X^{\frac1n}}
\Big [ \exp \Big\{ \lambda F^\gamma_{t}(X^{\frac1n}(\cdot n))
\Big\} \Big ]
= \E_0^X \big[ \exp \{ \lambda F_t^{\gamma}(X)\} \big].
\end{equation}
\end{lemma}

\begin{proof}
By coupling the successive non-trivial jumps of $X^{\frac1n}$ with those of $X$, it is easily seen that the local time process $(L^{X^{\frac1n}(\cdot n)}_t(x))_{x\in\Z}$ converges in distribution to $(L^{X}_t(x))_{x\in\Z}$, and hence $F^\gamma_{t}(X^{\frac1n}(\cdot n))$ also converges in distribution to $F_t^{\gamma}(X)$ as $n\to\infty$. Therefore to establish \eqref{discapprox}, it suffices to show that
$(\exp\{\lambda F^\gamma_{t}(X^{\frac1n}(\cdot n))\})_{n\in\N}$ are uniformly integrable.

Note that
\begin{equation}\label{34}
F^\gamma_{t}(X^{\frac1n}(\cdot n)) \leq |{\rm Range}_{s\in [0, nt]}(X^{\frac1n}(s))|
\end{equation}
is bounded by the number of non-trivial jumps of $X^{\frac1n}$ before time $nt$, which is a binomial random variable Bin$(nt, \kappa/n)$.
Since the exponential moment generating function of the sequence of Bin$(nt, \kappa/n)$ random variables converges to that of a Poisson random variable with mean $\kappa t$, the uniform integrability of $(\exp\{\lambda F^\gamma_{t}(X^{\frac1n}(\cdot n))\})_{n\in\N}$ then follows.
\end{proof}

We will also need the following result.
\begin{lemma}\label{lem:hitprob}
Let $X$ be a continuous time random walk on $\Z$ with jump rate $\kappa>0$, whose jump kernel has mean zero and variance $\sigma^2\in (0,\infty)$. Let $L^X_t(0)$ be its local time at $0$ by time $t$, and $\tau_0$ the first hitting time of $0$. Then there exists $C>0$ such that
\begin{equation}\label{EgLt}
\E^X_0\big[e^{-\gamma L_t^X(0)}\big] \sim \frac{\sigma}{\gamma} \sqrt{\frac{2\kappa}{\pi t}}
\quad \text{ as } t \to \infty \quad \mbox{and}\quad \P^X_z(\tau_0\geq t) \leq 1 \wedge \frac{C|z|}{\sqrt t} \ \forall\, t>0, z\in\Z.
\end{equation}
Furthermore, if $X^{\frac1n}(\cdot n)$ denote the random walks that approximate $X$ as in Lemma \ref{lem:discreteTimeApprox}, then there exists $C'>0$ such that for any $T>0$,
\begin{equation}\label{EgLt2}
\E^{X^{\frac1n}}_0\big[e^{-\frac{\gamma}{n} L_{nt}^{X^{\frac1n}}(0)}\big] \leq 1 \wedge \frac{C'}{\sqrt{t}} \quad \mbox{and} \quad \P^{X^{\frac1n}}_z(\tau_0\geq nt) \leq 1 \wedge \frac{C'|z|}{\sqrt t}
\end{equation}
uniformly in $t\in [0,T]$, $z\in\Z\backslash\{0\}$, and $n$ sufficiently large.
\end{lemma}
\begin{proof}
When $X$ is a continuous time simple symmetric random walk, the first part of \eqref{EgLt} was proved in \cite[Section 2.2]{DrGaRaSu-10} using the local central limit theorem and Karamata's Tauberian theorem. The same proof can also be applied to general $X$ with mean zero and finite variance. The second part of \eqref{EgLt} follows from Theorem 5.1.7 of \cite{LL10}.

By \eqref{EgLt},
$$
\E^X_0\big[e^{-\gamma L_t^X(0)}\big] \leq 1 \wedge \frac{C}{\sqrt t}
$$
for some $C$ uniformly in $t>0$. By the same reasoning as in the proof of Lemma \ref{lem:discreteTimeApprox}, $\E^X_0\big[e^{-\frac{\gamma}{n} L_{nt}^{X^{\frac1n}}(0)}\big]$ is a family of decreasing continuous functions in $t$ that converge pointwise to the continuous function $\E^X_0\big[e^{-\gamma L_t^X(0)}\big]$ as $n\to\infty$. Therefore this convergence must be uniform on $[0,T]$, which implies the first part of \eqref{EgLt2}. The second part of \eqref{EgLt2} follows by the same argument.
\end{proof}

\begin{proof}[Proof of Proposition \ref{prop:expMomLocal}]
We may restrict our attention to $t\geq t_0$ for some large $t_0$, since otherwise \eqref{prop3.1} is easily shown if we bound $F^\gamma_t(X)$
by the number of jumps of $X$ before time $t$.

Due to Lemma \ref{lem:discreteTimeApprox}, it then suffices to show that for some $C(\gamma)<\infty$ and for all $t\geq t_0$,
\begin{equation}
\lim_{n \to \infty} \E_0^{X^{\frac1n}}
\Big [ \exp \big\{ \lambda_t F^\gamma_{t}(X^{\frac1n}(\cdot \, n))
\big\} \Big ]
\le C(\gamma).
\end{equation}

Denote $L_{nt}^{(n)}(\cdot):=L_{nt}^{X^{\frac1n}}(\cdot)$ for simplicity, and for $x\in \Z$, let $\tau_x$ denote the first time $X^{\frac1n}$ visits $x$. By Taylor expansion and the definition of $F^\gamma_t$, we have
\begin{align}
\E_0^{X^{\frac1n}} \Big [ \exp \big\{ \lambda_t F^\gamma_{t}(X^{\frac1n}(\cdot \, n)) \big\} \Big ]
=\, & 1+ \!\!\sum_{k=1}^\infty \frac{\lambda_t^k}{k!} \sum_{x_1, \ldots, x_k\in\Z} \E_0^{X^\frac1n}\Big[\prod_{i=1}^k e^{-\frac{\gamma}{n} L_{nt}^{(n)}(x_i)} \ind{1}_{L_{nt}^{(n)}(x_i)>0}\Big] \nonumber \\
=\, & 1+ \!\!\sum_{k=1}^\infty \frac{\lambda_t^k}{k!} \sum_{x_1, \ldots, x_k\in\Z\atop 0\leq s_1,\ldots, s_k\leq nt} \E_0^{X^\frac1n}\Big[\prod_{i=1}^k e^{-\frac{\gamma}{n} L_{nt}^{(n)}(x_i)} \ind{1}_{\tau_{x_i}=s_i}\Big] \nonumber \\
\leq\, & 1+ \!\!\sum_{k=1}^\infty \lambda_t^k \sum_{m=1}^k \frac{m^{k-m}}{(k-m)!}
\sum_{0\leq t_1<t_2<\cdots<t_m\leq nt \atop y_1, \ldots, y_m\in\Z} \!\!\!\!\!\!\!\!\!\! \E_0^{X^\frac1n}\Big[\prod_{i=1}^m e^{-\frac{\gamma}{n} L_{nt}^{(n)}(y_i)} \ind{1}_{\tau_{y_i}=t_i}\Big]  \nonumber \\
=\, & 1+ \!\!\sum_{m=1}^\infty e^{m\lambda_t}\lambda_t^m \sum_{0\leq s_1<s_2<\cdots<s_m\leq nt \atop x_1, \ldots, x_m\in\Z} \!\!\!\!\!\!\!\!\!\! \E_0^{X^\frac1n}\Big[\prod_{i=1}^m e^{-\frac{\gamma}{n} L_{nt}^{(n)}(x_i)} \ind{1}_{\tau_{x_i}=s_i}\Big],  \label{3.1.2}
\end{align}
where in the inequality, we took advantage of the fact that for any $0\leq t_1<t_2<\cdots <t_m\leq nt$ with $1\leq m\leq k$, the number of ways of choosing $s_1, \ldots, s_k$ from $\{t_1, \ldots, t_m\}$ so that each $t_i$ is chosen at least once is given by $m!S(k,m)\leq \frac{1}{2}\frac{k!m^{k-m}}{(k-m)!}$, where $S(k,m)$ is called a Stirling number of the second kind~\cite[Theorem 3]{RD69}. We also used that when
$\tau_{x_i}=s_i=\tau_{x_j}=s_j$, we must have $x_i=x_j$.

Using $L^{(n)}_{nt}(x_i)\geq L^{(n)}_{s_m}(x_i)$ for $1\leq i\leq m-1$, and applying the strong Markov property at time $\tau_{x_k}=s_k$, we can bound the expectation in \eqref{3.1.2} by
\begin{align}
& \E_0^{X^\frac1n}\Big[\prod_{i=1}^{m-1} e^{-\frac{\gamma}{n} L_{s_m}^{(n)}(x_i)} \prod_{i=1}^{m} \ind{1}_{\tau_{x_i}=s_i}\Big] \E_0^{X^\frac1n}\Big[e^{-\frac{\gamma}{n} L_{nt-s_m}^{(n)}(0)}\Big] \nonumber \\
\leq\ & \E_0^{X^\frac1n}\Big[\prod_{i=1}^{m-1} e^{-\frac{\gamma}{n} L_{s_m}^{(n)}(x_i)} \prod_{i=1}^{m} \ind{1}_{\tau_{x_i}=s_i}\Big] \cdot C\phi\big(t-\frac{s_m}{n}\big), \label{3.1.3}
\end{align}
where $\phi(u):=1 \wedge \frac{1}{\sqrt{u}}$ and we applied \eqref{EgLt2} to obtain the inequality.

We now bound the expectation in \eqref{3.1.3}, summed over $x_m\in\Z$. Let $r := \lfloor\frac{s_{m-1} +s_m}{2}\rfloor$. Using $L^{(n)}_{s_m}(x_i)\geq L^{(n)}_r(x_i)$ for $1\leq i\leq m-1$ and applying the Markov property at time $r$ gives
\begin{eqnarray}
&& \sum_{ x_m \in \Z}
\E_0^{X^\frac1n} \Big[\prod_{i=1}^{m-1} e^{-\frac{\gamma}{n} L_{s_m}^{(n)}(x_i)} \prod_{i=1}^{m} \ind{1}_{\tau_{x_i}=s_i}\Big]
= \sum_{ x_m, y \in \Z}
\E_0^{X^\frac1n} \Big[\prod_{i=1}^{m-1} e^{-\frac{\gamma}{n} L_{s_m}^{(n)}(x_i)} \prod_{i=1}^{m} \ind{1}_{\tau_{x_i}=s_i} \cdot \ind{1}_{X^{\frac1n}(r)=y} \Big] \nonumber \\
&\le& \sum_{ y \in \Z} \E_0^{X^\frac1n} \Big[\prod_{i=1}^{m-1} e^{-\frac{\gamma}{n} L_{r}^{(n)}(x_i)} \prod_{i=1}^{m-1} \ind{1}_{\tau_{x_i}=s_i} \cdot \ind{1}_{X^{\frac1n}(r)=y} \Big]
\cdot \sum_{x_m \in \Z} \P_y^{X^\frac1n} \big(\tau_{x_m} = s_m -r \big). \label{3.1.4}
\end{eqnarray}
If $\widetilde X^{\frac1n}$ denotes the time-reversal of $X^{\frac1n}$, which has the same increment distribution as $-X^{\frac1n}$, then by time reversal and translation invariance, we have
\begin{eqnarray}
\sum_{x_m \in \Z} \P_y^{X^\frac1n} \big(\tau_{x_m} = s_m -r \big) &=& \P_0^{\widetilde X^{\frac1n}}\big(\widetilde X^{\frac1n}(1)\neq 0,\
\widetilde \tau_0> s_m-r\big)
= \frac{\kappa}{n} \sum_{z\in\Z} p_X(z) \P_z^{\widetilde X^{\frac1n}}(\tau_0 \geq s_m-r) \nonumber\\
&\leq& \frac{\kappa}{n} \sum_{z\in\Z} |z|p_X(z) \Big(1\wedge \frac{C'}{\sqrt{\frac{s_m}{n}-\frac{r}{n}}}\Big) \leq \frac{C}{n} \phi\big(\frac{s_m-s_{m-1}}{n}\big),
 \label{3.1.5}
\end{eqnarray}
where $\widetilde \tau_0:=\min\{i\geq 1: \widetilde X^{\frac1n}(i)=0\}$, and we applied \eqref{EgLt2} in the first inequality. Note that this bound no longer depends on $y$.

Substituting the bound of \eqref{3.1.5} into \eqref{3.1.4}, and then successively into \eqref{3.1.3} and \eqref{3.1.2}, we obtain
\begin{align}
& \sum_{x_1, \ldots, x_m\in\Z} \E_0^{X^\frac1n}\Big[\prod_{i=1}^m e^{-\frac{\gamma}{n} L_{nt}^{(n)}(x_i)} \ind{1}_{\tau_{x_i}=s_i}\Big] \nonumber\\
\leq\ \ &  \frac{C^2}{n} \phi\big(t-\frac{s_m}{n}\big) \phi\big(\frac{s_m-s_{m-1}}{n}\big)
\sum_{x_1, \ldots, x_{m-1}\in\Z}
\E_0^{X^\frac1n} \Big[\prod_{i=1}^{m-1} e^{-\frac{\gamma}{n} L_{r}^{(n)}(x_i)} \ind{1}_{\tau_{x_i}=s_i}\Big]. \label{312}
\end{align}
We can now iterate this bound to obtain
\begin{align}\label{313}
\sum_{x_1, \ldots, x_m\in\Z} \E_0^{X^\frac1n}\Big[\prod_{i=1}^m e^{-\frac{\gamma}{n} L_{nt}^{(n)}(x_i)} \ind{1}_{\tau_{x_i}=s_i}\Big]
\leq \frac{C^m}{n^m} \phi\big(\frac{s_1}{n}\big) \prod_{i=2}^m \phi^2\big(\frac{s_i-s_{i-1}}{n}\big) \cdot \phi\big(t-\frac{s_m}{n}\big),
\end{align}
where $\phi(u)=1\wedge \frac{1}{\sqrt u}$. Therefore the inner summand in \eqref{3.1.2} can be bounded by

\begin{eqnarray}
&&\sum_{0 \le s_1 < s_2 < \ldots < s_m \le nt}
\frac{C^m}{n^m} \phi\big(\frac{s_1}{n}\big) \prod_{i=2}^m \phi^2\big(\frac{s_i-s_{i-1}}{n}\big) \cdot \phi\big(t-\frac{s_m}{n}\big) \nonumber\\
&\le&  C^m \idotsint\limits_{0<t_1<\cdots <t_m<t} \phi(t_1)\phi(t-t_m)\prod_{i=2}^m \phi^2(t_i-t_{i-1}) {\rm d} t_1\cdots {\rm d}t_m. \label{eq:trafo}
\end{eqnarray}
Note that given $t_{j-1}<t_{j+1}$,
\begin{align}
\int_{t_{j-1}}^{t_{j+1}}
\phi^2(t_j-t_{j-1}) \phi^2(t_{j+1}-t_{j}) {\rm d}t_j & = \int_{t_{j-1}}^{t_{j+1}}
\Big(1\wedge \frac{1}{t_j-t_{j-1}}\Big) \Big(1\wedge \frac{1}{t_{j+1}-t_{j}}\Big) {\rm d}t_j \nonumber \\
& \leq 4(\ln t) \Big( 1\wedge \frac{1}{t_{j+1}-t_{j-1}}\Big) = 4(\ln t) \phi^2(t_{j+1}-t_{j-1}), \label{eq:singleFactor}
\end{align}
where the bound clearly holds when $t_{j+1}-t_{j-1}\leq 1$. When $t_{j+1}-t_{j-1}>1$, the inequality is obtained by
dividing the interval of integration into $[t_{j-1}, (t_{j+1}-t_{j-1})/2]$ and $[(t_{j+1}-t_{j-1})/2, t_{j+1}]$, where
in the first case we use the bound $\frac{1}{t_{j+1}-t_j} \leq \frac{2}{t_{j+1}-t_{j-1}}$, and in the second case we use
the bound $\frac{1}{t_{j}-t_{j-1}} \leq \frac{2}{t_{j+1}-t_{j-1}}$.

Applying \eqref{eq:singleFactor} repeatedly to \eqref{eq:trafo} to integrate out $t_2, \ldots, t_{m-1}$, we can bound the right-hand side of
\eqref{eq:trafo} from above by
\begin{equation}\label{316}
C^m (4\ln t)^{m-2} \iint\limits_{0 < t_1  < t_m < t}
 \Big(1\wedge \frac{1}{\sqrt t_1}\Big)
\Big(1 \wedge \frac{1}{ t_{m} - t_1} \Big) \Big(1\wedge \frac{1}{\sqrt{t-t_m}}\Big)\, {\rm d}t_1 {\rm d}t_m
\le \widetilde C^m (\ln t)^{m-1},
\end{equation}
where the integral is bounded by considering the three cases: $t_1\geq t/3$, $t_m-t_1\geq t/3$, or $t-t_m\geq t/3$.
Substituting this bound for \eqref{eq:trafo} back into \eqref{3.1.2} then gives
\begin{align*}
\E_0^{X^\frac1n} \big[ \exp \{ \lambda_t F^\gamma_{t}(X^{\frac1n}(\cdot n))\} \big]
&\le 1+\sum_{m=1}^\infty  e^{m\lambda_t}\lambda_t^m\widetilde C^m (\ln t)^{m-1}
\le \frac{1}{1-e^{c(\gamma)}c(\gamma)\widetilde C} =:C(\gamma) <\infty
\end{align*}
uniformly in $n$ if $c(\gamma)$ is chosen small enough such that $e^{c(\gamma)}c(\gamma)<1/\widetilde C$. This finishes the proof.
\end{proof}

\section{Proof of Theorem \ref{thm:range}} \label{prange}
The proof follows the same line of argument as that of Proposition \ref{prop:expMomLocal}, except for some complications. We first approximate $X$ by the family of discrete time
random walks $X^{\frac1n}(\cdot n)$, $n\in\N$. Recall from \eqref{GtX} that
\begin{equation}\label{GtX2}
G_t(X):= \big(\sup_{0\leq s\leq t} X_s - \inf_{0\leq s\leq t} X_s\big) - |{\rm Range}_{s\in [0,t]}(X_s)| = \sum_{\inf_{s\in [0,t]}X_s <x< \sup_{s\in [0,t]} X_s} \ind{1}_{L^X_t(x)=0}.
\end{equation}
The following is an analogue of Lemma \ref{lem:discreteTimeApprox}.

\begin{lemma} \label{lem:dTA}
Let $X$ be a continuous time random walk on $\Z$, whose jump kernel $p_X$ satisfies $\sum_{x\in\Z} p_X(x)e^{\lambda_* |x|}<\infty$ for some $\lambda_*>0$. Let $X^{\frac1n}(\cdot n)$ be the discrete time approximation of $X$ defined as in \eqref{32}. Then for any $\lambda<\lambda_*$ and $t \in [0,\infty)$, we have
\begin{equation}\label{discapprox2}
\lim_{n \to \infty} \E_0^{X^{\frac1n}} \Big [ \exp \Big\{ \lambda G_{t}(X^{\frac1n}(\cdot n)) \Big\} \Big ]
= \E_0^X \big[ \exp \{ \lambda G_t(X)\} \big].
\end{equation}
\end{lemma}
\begin{proof}
Clearly $G_t(X^{\frac{1}{n}}(\cdot n))$ converges in distribution to $G_t(X)$ as $n\to\infty$. It remains to show the uniform integrability of
$(\exp\{\lambda G_{t}(X^{\frac1n}(\cdot n))\})_{n\in\N}$. Similarly, as in  \eqref{34} we have,
$$
G_t(X^{\frac{1}{n}}(\cdot n)) \leq \sup_{0\leq i\leq nt}X^{\frac{1}{n}}(i) - \inf_{0\leq i\leq nt}X^{\frac{1}{n}}(i)
$$
is bounded by the sum of the sizes of the jumps of $X^{\frac1n}$ before time $nt$, which is a compound binomial random variable with binomial
parameters $(nt, \kappa/n)$ and summand distribution $\widehat p_X(x)=p_X(x) 1_{x\geq 0} +p_X(-x)1_{x>0}$. As $n\to\infty$, this converges to a compound Poisson random variable with Poisson parameter $\kappa t$ and summand distribution $\widehat p_X$. Since we assume $\sum_{x\in\Z}e^{\lambda_*|x|} p_X(x)<\infty$ for some $\lambda_*>0$, it is then easily seen that $(\exp\{\lambda G_{t}(X^{\frac1n}(\cdot n))\})_{n\in\N}$ is uniformly integrable for $\lambda <\lambda_*$.
\end{proof}

\begin{proof}[{Proof of Theorem \ref{thm:range}}] As in the proof of Proposition \ref{prop:expMomLocal}, it suffices to show that for some $C<\infty$ and for all $t$ sufficiently large,
\begin{equation}
\lim_{n \to \infty} \E_0^{X^{\frac1n}} \Big [ \exp \big\{ \lambda_t G_{t}(X^{\frac1n}(\cdot \, n)) \big\} \Big ] \le C.
\end{equation}

Given $X^{\frac1n}(0)=0$, for $x\in\Z$, define
$$
\widetilde \tau_x:= \left\{
\begin{aligned}
& \min\{i\geq 0: X^{\frac1n}(i) \geq x\} \qquad \mbox{if } x\geq 0 \\
& \min\{i\geq 0: X^{\frac1n}(i) \leq x\} \qquad \mbox{if } x\leq 0
\end{aligned}
\right.,
\qquad \tau_x:=\min\{i\geq 0: X^{\frac1n}(i)=x\}.
$$
Using \eqref{GtX2}, as in \eqref{3.1.2}, we can expand
\begin{align}
& \E_0^{X^{\frac1n}} \Big [ \exp \big\{ \lambda_t G_{t}(X^{\frac1n}(\cdot \, n)) \big\} \Big ] \nonumber \\
=\ & 1+ \!\!\sum_{k=1}^\infty \frac{\lambda_t^k}{k!} \sum_{x_1, \ldots, x_k\in\Z} \E_0^{X^\frac1n}\Big[\prod_{i=1}^k  \ind{1}_{\widetilde\tau_{x_i}\leq nt<\tau_{x_i}}\Big]
=\ 1+ \!\!\sum_{k=1}^\infty \frac{\lambda_t^k}{k!} \sum_{x_1, \ldots, x_k\in\Z\atop 0\leq s_1,\ldots, s_k\leq nt} \E_0^{X^\frac1n}\Big[\prod_{i=1}^k  \ind{1}_{\widetilde\tau_{x_i}=s_i, \tau_{x_i}>nt}\Big] \nonumber \\
=\ & 1+ \!\!\sum_{k=1}^\infty \frac{\lambda_t^k}{k!} \sum_{m=1}^k \sum_{0< t_1<\cdots<t_m\leq nt \atop I_1, \ldots, I_m \vdash\{1,\ldots, k\}} \sum_{x_1, \ldots, x_k\in\Z}  \E_0^{X^\frac1n}\Big[\prod_{i=1}^m \prod_{j\in I_i} \ind{1}_{\widetilde\tau_{x_j}=t_i, \tau_{x_j}>nt}\Big] \nonumber\\
=\ & 1+ \!\!\sum_{k=1}^\infty \frac{\lambda_t^k}{k!} \sum_{m=1}^k \sum_{0< t_1<\cdots<t_m\leq nt \atop I_1, \ldots, I_m \vdash\{1,\ldots, k\}} \sum_{\substack{x_1,\ldots, x_k\in\Z\\ y_1,\ldots, y_m\in\Z \\ z_1, \ldots, z_m\in\Z}} \!\!\!\!\!\! \E_0^{X^\frac1n}\Big[\prod_{i=1}^m \Big(\ind{1}_{X^{\frac1n}(t_i-1)=y_i, X^{\frac1n}(t_i)=z_i}\prod_{j\in I_i} \ind{1}_{\widetilde\tau_{x_j}=t_i, \tau_{x_j}>nt}\Big)\Big], \label{44}
\end{align}
where in the third line, we summed over all ordered non-empty disjoint sets $(I_1, \ldots, I_m)$ which partition $\{1,\ldots, k\}$. 
Note that when $\widetilde \tau_{x_j}=t_m$ for all $j\in I_m$, $x_j$ must be strictly between $y_m$ and $z_m$ for all $j\in I_m$.
By the Markov property at time $t_m$ and by Lemma \ref{lem:hitprob}, we can bound 
\begin{align}
& \E_0^{X^\frac1n}\Big[\prod_{i=1}^m \Big(\ind{1}_{X^{\frac1n}(t_i-1)=y_i, X^{\frac1n}(t_i)=z_i}\prod_{j\in I_i} \ind{1}_{\widetilde\tau_{x_j}=t_i, \tau_{x_j}>nt}\Big)\Big] \nonumber\\
\leq \ & \E_0^{X^\frac1n}\Big[\prod_{i=1}^{m-1} \Big(\ind{1}_{X^{\frac1n}(t_i-1)=y_i, X^{\frac1n}(t_i)=z_i}\prod_{j\in I_i} \ind{1}_{\widetilde\tau_{x_j}=t_i, \tau_{x_j}>t_m}\Big)\Big(\prod_{j\in I_m} \ind{1}_{\tau_{x_j}>t_m}\Big) \ind{1}_{X^{\frac1n}(t_m-1)=y_m}\Big] \nonumber \\
& \qquad \qquad \quad \times \frac{\kappa}{n} p_X(z_m-y_m)\, \max_{x\in (y_m\wedge z_m, y_m\vee z_m)} \P^{X^{\frac1n}}_{z_m}(\tau_x>nt-t_m) \cdot \prod_{j\in I_m} \ind{1}_{x_j\in (y_m\wedge z_m, y_m\vee z_m)} \nonumber \\
\leq \ & \E_0^{X^\frac1n}\Big[\prod_{i=1}^{m-1} \Big(\ind{1}_{X^{\frac1n}(t_i-1)=y_i, X^{\frac1n}(t_i)=z_i}\prod_{j\in I_i} \ind{1}_{\widetilde\tau_{x_j}=t_i, \tau_{x_j}>t_m}\Big)\Big(\prod_{j\in I_m} \ind{1}_{\tau_{x_j}>t_m}\Big) \ind{1}_{X^{\frac1n}(t_m-1)=y_m}\Big] \nonumber \\
& \qquad \qquad \quad \times \frac{\kappa}{n} p_X(z_m-y_m)\, C|z_m-y_m|\phi\big(t-\frac{t_m}{n}\big) \cdot \prod_{j\in I_m} \ind{1}_{x_j\in (y_m\wedge z_m, y_m\vee z_m)}, \label{45}
\end{align}
where as before $\phi(u) = 1\wedge \frac{1}{\sqrt u}$, and this is the analogue of \eqref{3.1.3} in the proof of Proposition \ref{prop:expMomLocal}. 

Let $r := \lfloor\frac{t_{m-1} +t_m}{2}\rfloor$. Applying the Markov property at time $r$ and summing the above bound over $y_m,z_m$ and $(x_j)_{j\in I_m}$ then gives
\begin{align}
& C\frac{\kappa}{n}\phi\big(t-\frac{t_m}{n}\big) \!\!\!\!\!\!\!\!\!\!\!\! \sum_{\substack{y_m,z_m\in\Z \\ x_j\in (y_m\wedge z_m, y_m\vee z_m):\, j\in I_m}}\!\!\!\!\!\!\!\!\!\!\!\!\!\!\! p_X(z_m-y_m) |z_m-y_m| \sum_{w\in\Z} \P^{X^{\frac1n}}_w \Big[ \Big(\prod_{j\in I_m} \ind{1}_{\tau_{x_j}\ge t_m-r}\Big) \ind{1}_{X^{\frac1n}(t_m-1-r)=y_m}\Big]\nonumber \\
& \qquad \qquad \qquad \qquad \quad \times \E_0^{X^\frac1n}\Big[\prod_{i=1}^{m-1} \Big(\ind{1}_{X^{\frac1n}(t_i-1)=y_i, X^{\frac1n}(t_i)=z_i}\prod_{j\in I_i} \ind{1}_{\widetilde\tau_{x_j}=t_i, \tau_{x_j}>r}\Big)\ind{1}_{X^{\frac1n}(r)=w}\Big] \nonumber \\
\leq\ & C\frac{\kappa}{n}\phi\big(t-\frac{t_m}{n}\big) \!\!\! \sum_{v, w\in\Z}\!\!\! p_X(v) |v|^{|I_m|+1}  \max_{x\in (0\wedge v, 0\vee v)} \P^{\widetilde X^{\frac1n}}_0 (\tau_x \ge t_m-r) \nonumber \\
& \qquad \qquad \qquad \qquad \quad \times \E_0^{X^\frac1n}\Big[\prod_{i=1}^{m-1} \Big(\ind{1}_{X^{\frac1n}(t_i-1)=y_i, X^{\frac1n}(t_i)=z_i}\prod_{j\in I_i} \ind{1}_{\widetilde\tau_{x_j}=t_i, \tau_{x_j}>r}\Big)\ind{1}_{X^{\frac1n}(r)=w}\Big] \nonumber \\
\leq\ & C^2\frac{\kappa}{n}\phi\big(t-\frac{t_m}{n}\big)\phi\big(\frac{t_m-t_{m-1}}{2}\big) \!\! \sum_{v\in\Z}\!\! p_X(v) |v|^{|I_m|+2} \nonumber \\
& \qquad \qquad \qquad \qquad \quad \times \E_0^{X^\frac1n}\Big[\prod_{i=1}^{m-1} \Big(\ind{1}_{X^{\frac1n}(t_i-1)=y_i, X^{\frac1n}(t_i)=z_i}\prod_{j\in I_i} \ind{1}_{\widetilde\tau_{x_j}=t_i, \tau_{x_j}>r}\Big)\Big], \label{46}
\end{align}
where we have reversed time for $X^{\frac1n}$ on the time interval
$[r, t_m-1]$, with $\widetilde X^{\frac1n}$ denoting the time-reversed
random walk, and in the last inequality we again applied Lemma
\ref{lem:hitprob}. This bound is the analogue of \eqref{312}, which
can now be iterated. The calculations in \eqref{313}--\eqref{316} then
give
\begin{align}
& \sum_{0< t_1<\cdots<t_m\leq nt} \sum_{\substack{x_1,\ldots, x_k\in\Z\\ y_1,\ldots, y_m\in\Z \\ z_1, \ldots, z_m\in\Z}} \!\! \E_0^{X^\frac1n}\Big[\prod_{i=1}^m \Big(\ind{1}_{X^{\frac1n}(t_i-1)=y_i, X^{\frac1n}(t_i)=z_i}\prod_{j\in I_i} \ind{1}_{\widetilde\tau_{x_j}=t_i, \tau_{x_j}>nt}\Big)\Big] \nonumber \\
\leq \ & \widetilde C^m (\ln t)^{m-1} \prod_{i=1}^m M(|I_i|+2), \label{47}
\end{align}
where $M(\alpha):=\sum_{x\in\Z} |x|^\alpha p_X(x)$. Substituting this bound into \eqref{44} then gives (uniformly in $n$) 
\begin{align}
& \E_0^{X^{\frac1n}} \Big [ \exp \big\{ \lambda_t G_{t}(X^{\frac1n}(\cdot \, n)) \big\} \Big ] 
\leq  1+ \!\!\sum_{k=1}^\infty \frac{\lambda_t^k}{k!} \sum_{m=1}^k \sum_{I_1, \ldots, I_m \vdash\{1,\ldots, k\}} \widetilde C^m (\ln t)^{m-1} \prod_{i=1}^m M(|I_i|+2) \nonumber\\
\leq\ & 1+ \!\!\sum_{m=1}^\infty (\widetilde C \ln t)^{m} \sum_{k_1, \ldots, k_m=1}^\infty \frac{\lambda_t^{k_1+\cdots +k_m}}{(k_1+\cdots +k_m)!}\sum_{I_1, \ldots, I_m \vdash\{1,\ldots, k_1+\cdots +k_m\}\atop |I_1|=k_1, \ldots, |I_m|=k_m} \prod_{i=1}^m M(k_i+2) \nonumber \\
=\ & 1+ \!\!\sum_{m=1}^\infty (\widetilde C\ln t)^{m} \sum_{k_1, \ldots, k_m=1}^\infty \frac{\lambda_t^{k_1+\cdots +k_m}}{(k_1+\cdots +k_m)!}
{k_1+\cdots+k_m \choose k_1, \ldots, k_m} \prod_{i=1}^m M(k_i+2) \nonumber \\
=\ & 1+ \!\!\sum_{m=1}^\infty (\widetilde C\ln t)^{m} \Big(\sum_{k=1}^\infty \frac{\lambda_t^k}{k!}M(k+2)\Big)^m 
= \frac{1}{1- \widetilde C\ln t \sum_{k=1}^\infty \frac{\lambda_t^k}{k!}M(k+2)} <C<\infty \label{48}
\end{align}
if $c$ in $\lambda_t=\frac{c}{1\vee \ln t}$ is chosen small enough. Indeed, let $V$ be a random variable with $M(\alpha)=\E[|V|^\alpha]$. Since assumption (\ref{expmomy}) implies $\E[e^{\lambda_* |V|}]<\infty,$ we have
\begin{align*}
& \widetilde C\ln t \sum_{k=1}^\infty \frac{\lambda_t^k}{k!}M(k+2)  =\E\Big[\sum_{k=1}^\infty  \frac{c^k\widetilde C\ln t }{(1\vee\ln t)^k k!} |V|^{k+2} \Big] \\
\leq\ & c\,\widetilde C\, \E\Big[|V|^3   \sum_{k=1}^\infty \frac{c^{k-1}}{(1\vee\ln t)^{k-1} (k-1)!} |V|^{k-1} \Big]
\leq c\,\widetilde C\, \E\big[|V|^6\big]^{\frac12}\E\Big[e^{\frac{2c}{1\vee\ln t} |V|}\Big]^{\frac12} <1
\end{align*}
if $t$ is large enough and $c$ is chosen small enough. This completes the proof. 
\end{proof}

\section{Proof of Theorem \ref{thm:main}} \label{pmain}

To prepare for the proof, we recall here the strategy for the lower
bound on the annealed survival probability employed in
\cite{DrGaRaSu-10}, and we show how to rewrite the survival
probability in terms of the range of random walks.

\subsection{Strategy for lower bound} \label{subsec:Previous}

The lower bound on the annealed survival probabilities in Theorem \ref{thm:annealedAsymptotics} follows the same strategy as for the case of immobile traps in previous works. Denote $B_r=\{x\in\Z^d : \Vert x\Vert_\infty\leq r\}$. For a fixed time $t$, we force the environment $\xi$ to create a ball $B_{R_t}$ of radius $R_t$ around the origin, which is free of traps from time $0$ up to time $t$. We then force the random walk $X$ to stay inside $B_{R_t}$ up to time $t$. This leads to a lower bound on the survival probability that is independent of $\gamma\in (0,\infty]$.

To be more precise, we consider the following events:
\begin{itemize}
 \item
 Let $E_t$
denote the event that $N_y=0$ for all $y\in B_{R_t}$.
\item
 Let $F_t$ denote the event that $Y^{j,y}_s \notin B_{R_t}$ for all
$y\notin B_{R_t},$ $1\leq j\leq N_y,$ and $s \in [0,t].$
\item
Let $G_t$ denote the event that $X$ with $X(0)=0$ does not leave $B_{R_t}$ before time $t$.
\end{itemize}
Then, by the strategy outlined above, the annealed survival probability
\begin{equation}\label{eq:lowerBdSurvProbImm}
\E^X_0[Z^\gamma_{t,X}] \geq \P(E_t \cap F_t \cap G_{t}) = \P(E_t)\P(F_t)\P(G_{t}),
\end{equation}
since $E_t,$ $F_t,$ and $G_t$ are independent.

In order to lower bound \eqref{eq:lowerBdSurvProbImm}, note that
\begin{equation} \label{eq:EtProb}
\P(E_t) = e^{-\nu (2R_t+1)^d}.
\end{equation}

To estimate $\P(G_{t})$, Donsker's invariance principle implies that there exists $\alpha>0$ such that for all $t$ sufficiently large,
\begin{equation*}
\inf_{x\in B_{\sqrt t/2}}\P^X_0 \left(X_s \in B_{\sqrt t}\ \forall\ s\in [0,t]\, ,\ X_t \in B_{\sqrt t/2} \; \Big|
\; X_0=x\right) \ge \alpha .
\end{equation*}
Now if $1 \ll R_t \ll \sqrt{t}$ as $t\to\infty$,
then by partitioning the time interval $[0,t]$ into intervals of length $R_t^2$ and applying the Markov property at times $iR_t^2$, we obtain
\begin{eqnarray} \label{eq:lowerGtEstimate}
\P(G_{t}) &\geq& \P_0^X \Big(X_s\in B_{R_t}\ \forall\ s\in [(i-1)R_t^2, iR_t^2], \text{ and } X_{i R_t^2}\in B_{R_t/2},\ i=1, 2, \cdots,
\lceil t/R_t^2\rceil\Big) \nonumber\\
&\geq& \alpha^{\lceil t/R_t^2\rceil} = (1+o(1))e^{t\ln \alpha/R_t^2}.
\end{eqnarray}
This actually gives the correct logarithmic order of decay for $\P(G_t)$. Indeed, by Donsker's invariance principle, uniformly in $t$ large
and $X_0=x\in B_{R_t}$,
$$
\P^X_x(X_s\notin B_{R_t} \mbox{ for some } s\in [0, R_t^2]) \geq \P^W_0(W_s \notin B_{3} \mbox{ for some } s\in [0,1])=: \rho>0,
$$
where $W$ is a standard $d$-dimensional Brownian motion. Therefore by a similar application of the Markov inequality as in \eqref{eq:lowerGtEstimate}, we find that
\begin{align} \label{eq:GtUpperBd}
\P(G_t) \le e^{\floor{t/R_t^2} {\ln (1-\rho)}}.
\end{align}

In dimension $d=1,$ which is our main focus, integrating out the Poisson initial distribution of the $Y$-particles gives
\begin{align} \label{eq:sFtAsympt}
\begin{split}
\P(F_t) &= \exp \Big \{
-\nu \sum_{y \in \Z \setminus B_{R_t} } \P^{Y}_{y} (\tau^Y(B_{R_{t}}) \leq t)
\Big \}
= \exp \Big \{
-\nu \sum_{y \in \Z \setminus \{0\} } \P^{Y}_{y} (\tau^Y(\{0\}) \leq t)
\Big \}\\
&= \exp \Big \{
-\nu \sum_{y \in \Z \setminus \{0\} } \P^{Y}_{0} (\tau^Y(\{-y\}) \leq t)
\Big \} = \exp \Big\{
-\nu  \big( \E^Y_0\big[\,\big|{\rm Range}_{s\in [0,t]}(Y_s)\big|\,\big]-1 \big)
\Big\} ,
\end{split}
 \end{align}
where $\tau^Y(B)$ denotes the first hitting time of a set $B \subset \Z^d$ by $Y$, and in the second equality, we used the assumption that $Y$ makes nearest-neighbor jumps. Note that it was shown in \cite{DrGaRaSu-10} that $-\ln \P(F_t)\sim \nu \sqrt{\frac{8\rho t}{\pi}}$.

Substituting the bounds \eqref{eq:EtProb}--\eqref{eq:sFtAsympt} into \eqref{eq:lowerBdSurvProbImm}, we find that in dimension $d=1$, the optimal choice is $R_t = t^\frac13$, which is determined by the interplay between $\P(E_t)$ and $\P(G_t)$ as $t \to \infty$. If this lower bound strategy
is optimal, then under $P^\gamma_t$, $X$ will fluctuate on the scale of $t^{\frac13}$.


\subsection{Rewriting in terms of the range}

Averaging out the Poisson initial condition of $\xi$,
we can rewrite \eqref{eq:condAnnProb} as
\begin{equation} \label{eq:integratedXiSol}
Z^\gamma_{t,X}
= \exp \Big\{ \nu \sum_{y \in \Z}  (  v_X(t,y)
- 1 ) \Big\},
\end{equation}
with
\begin{equation*}
 v_X(t,y) = \E_y^Y \Big[ \exp \Big\{ -\gamma L^{Y-X}_t(0)\Big\} \Big],
\end{equation*}
where $L^{Y-X}_t(0)=\int_0^t \delta_0 (Y_s - X_{s}) \, {\rm d}s$ is the local time of $Y-X$ at $0$, introduced in \eqref{loct}.


When $\gamma=\infty$, we define $v_X(t,Y) = \P^Y_y(L^{Y-X}_t(0)=0)$, and it is easily seen that
\begin{equation}
\begin{aligned}
Z^\infty_{t, X} = \exp \Big\{-\nu \sum_{y \in \Z} \P^Y_y\big(L^{Y-X}_t(0)>0\big)\Big\} & = \exp\big\{-\nu \E^Y_0\big[\,\big|{\rm Range}_{s\in [0,t]}(Y_s -X_s)\big|\,\big] \big\} \\
& = \exp\big\{-\nu \E^Y_0\big[\,\big|{\rm Range}_{s\in [0,t]}(Y_s
 +X_s)\big|\,\big] \big\},
\end{aligned}
\end{equation}
where we used the assumption that $Y$ is symmetric, and for any $f\in D([0,t], \Z)$,
\begin{equation}\label{range}
{\rm Range}_{s\in [0,t]}(f(s)) := \{ f(s) \, : \, s \in [0,t]\}
\end{equation}
denotes the range.

When $\gamma<\infty$, $Z^\gamma_{t, X}$ admits a similar representation in terms of the range of $Y+X$. Indeed, let $N_t:=\{J_1<J_2<\cdots\}$
be an independent Poisson point process on $[0,\infty)$ with rate $\gamma \in (0,\infty)$, and define
\begin{equation}\label{softrange}
{\rm SoftRange}_{s\in [0,t]}(f(s)) := \{ f(J_k) \, : \, k \in \N, \, J_k \in [0,t]\}.
\end{equation}
Probability and expectation for $N$ will be denoted by adding the superscript $N$ to $\P$ and $\E$. We can then rewrite \eqref{eq:integratedXiSol} as
\begin{equation}\label{ZgammaRep}
\begin{aligned}
Z^\gamma_{t,X} & = \exp\Big\{-\nu \sum_{y\in\Z} \P^{Y,N}_y
\big(0\in {\rm SoftRange}_{s\in [0,t]}(Y_s-X_s) \big)\Big\} \\
& = \exp\big\{-\nu \E^{Y,N}_0\big[\,\big|{\rm SoftRange}_{s\in [0,t]}(Y_s+X_s)\big|\,\big] \big\}.
\end{aligned}
\end{equation}

\subsection{Proof of Sub-diffusivity of X}


To control the path measure $P_t^\gamma$ (cf.~\eqref{eq:Gibbs}) we
could try to proceed with the bounds outlined in
\eqref{eq:lowerBdSurvProbImm}--\eqref{eq:sFtAsympt}. We cannot use
\eqref{eq:sFtAsympt} directly as we had assumed that $Y$ was a simple
random walk for that particular bound. To circumvent this we use
Theorem \ref{thm:range}.  However note that, we are free to use the bounds 
\eqref{eq:lowerBdSurvProbImm}--\eqref{eq:GtUpperBd} for $Y$ as in (\ref{grwy}).  Using these with
$R_t=t^{\frac13}$ and  \eqref{ZgammaRep}  we observe that
\begin{align} \label{interbound}
&P_{t}^{\gamma} (X\in \cdot ) \leq  \frac{\E^X_0\Big[\exp\big\{-\nu \E^{Y,N}_0\big[\,\big|{\rm SoftRange}_{s\in [0,t]}(Y_s+X_s)\big|\,\big] \big\} \ind{1}_{X\in \cdot} \Big]}{e^{-c t^{\frac{1}{3}}} \P(F_t)} 
\end{align}
 


When $Y$ is as in (\ref{grwy}) and satisfies (\ref{expmom}) then using
Theorem \ref{thm:range} we have
\begin{align*}
\P(F_t) & = \exp \Big \{
-\nu \sum_{y \in \Z \setminus B_{R_t} } \P^{Y}_{y} (\tau^Y(B_{R_{t}}) \leq t)
\Big \}
= \exp \Big \{ -\nu \sum_{y \in \Z \setminus B_{R_t}} \P^{Y}_{0} (\tau^Y(y+B_{R_t}) \leq t) \Big \} \\
& \ge \exp \Big \{ -\nu \Big(\E^Y_0\Big[\sup_{s\in [0,t]}Y_s -\inf_{s\in [0,t]}Y_s\Big] \Big) \Big \}
\ge \exp \Big \{ -\nu \Big(\E^Y_0\Big[ \Big| {\rm Range}_{s\in [0,t]}(Y_s)\Big|\Big] +C\ln t\Big) \Big \},
\end{align*}
where we used translation invariance and \eqref{EGt}. Since $\ln t\ll
t^{\frac13}$, this  and (\ref{interbound}) implies that 
\begin{eqnarray}
\lefteqn{P_{t}^{\gamma} (X\in \cdot ) \leq }  \label{eq:gibbsMeasRangeEst}\\&&\ e^{c_1 t^{\frac{1}{3}}} \E^X_0\Big[\exp\Big\{-\nu \Big(\E^{Y,N}_0\big[\,\big|{\rm SoftRange}_{s\in [0,t]}(Y_s+X_s)\big|\,\big]- \E^Y_0\big[\,\big|{\rm Range}_{s\in [0,t]}(Y_s)\big|\,\big]\Big) \Big\} \ind{1}_{X\in \cdot} \Big] \nonumber
\end{eqnarray}
for $t$ sufficiently large.  This  will be the starting point of our analysis of $P^\gamma_t$.

\begin{proof}[{Proof of Theorem \ref{thm:main} for simple random walks}]
We first bound the fluctuation of $X$ under $P^\gamma_t$ from below. Since $Y$ is an irreducible symmetric random walk,
Pascal's principle (see \cite[Prop. 2.1]{DrGaRaSu-10}, in particular,
 \cite[(38) \& (49)]{DrGaRaSu-10} for $\gamma < \infty$), implies that
the expected (soft) range (cf.\ \eqref{range}-\eqref{softrange}) of a $Y$ walk increases under perturbations, i.e., $\P^X_0$-a.s.,
\begin{equation} \label{eq:rangeDiff}
\E^{Y,N}_0\big[\,\big|{\rm SoftRange}_{s\in [0,t]}(Y_s+X_s)\big|\,\big]- \E^{Y,N}_0\big[\,\big|{\rm SoftRange}_{s\in [0,t]}(Y_s)\big|\,\big] \ge 0.
\end{equation}
Also note that by the definition of $F^\gamma_t(\cdot)$ in \eqref{eq:localTimeFunc} and the definition of soft range in \eqref{softrange}, we have
\begin{equation}\label{RsoR}
F_t^{\gamma}(Y) = \vert {\rm Range}_{s\in [0,t]}(Y_s)  \vert
- \E^{N}\big[\vert {\rm SoftRange}_{s\in [0,t]}(Y_s)\vert \big].
\end{equation}
Furthermore, Proposition \ref{prop:expMomLocal} applied to $Y$, combined with the exponential Markov inequality, gives
\begin{equation}\label{EFgamma}
\E_0^Y[ F_t^{\gamma} (Y)] = \int_0^\infty \P^Y_0(F_t^{\gamma}(Y) \geq m)\, {\rm d}m
 \le \int_0^\infty C(\gamma) e^{-\frac{m c(\gamma)}{1\vee \ln t}} {\rm d} m \le C\ln t
\end{equation}
for all $t$ large enough.
Hence, in combination with \eqref{eq:gibbsMeasRangeEst}, we obtain
\begin{align}
& P^\gamma_t \big( \Vert X  \Vert_t \le \alpha t^\frac13 \big) \nonumber\\
\le\ &
e^{c_1 t^{\frac{1}{3}}} \E^X_0\Big[\exp\Big\{-\nu \Big(\E^{Y,N}_0\big[\,\big|{\rm SoftRange}_{s\in [0,t]}(Y_s+X_s)\big|\,\big]- \E^Y_0\big[\,\big|{\rm Range}_{s\in [0,t]}(Y_s)\big|\,\big]\Big) \Big\}
\ind{1}_{\Vert X  \Vert_t \le \alpha t^\frac13}  \Big] \nonumber \\
\le\ &
e^{c_1 t^{\frac{1}{3}}} \E^X_0\Big[\exp\Big\{-\nu \Big(
\underbrace{\E^{Y,N}_0\big[\,\big|{\rm SoftRange}_{s\in [0,t]}(Y_s+X_s)\big|\,\big]- \E^{Y,N}_0\big[\,\big|{\rm SoftRange}_{s\in [0,t]}(Y_s)\big|\,\big]}_{
\stackrel{\eqref{eq:rangeDiff}}{\ge} 0} \Big) \Big\} \nonumber \\
& \qquad \qquad  \times \exp \big\{ \nu \E_0^Y[ F_t^{\gamma} (Y)] \big\}
\ind{1}_{\Vert X  \Vert_t \le t^\frac13}  \Big] \nonumber\\
\le\ &
e^{c_1 t^{\frac{1}{3}}} e^{\nu \E_0^Y[ F_t^{\gamma} (Y)]}\P^X_0 \Big( \Vert X \Vert_t
\le \alpha t^\frac13 \Big) \ \le\   e^{c_1 t^{\frac{1}{3}} + C\nu\ln t}  e^{\alpha^{-2} \ln (1-\rho)\, t^{\frac13}} \to 0 \quad \text{ as } t \to \infty, \label{3.19}
\end{align}
where we applied \eqref{eq:GtUpperBd} in the last inequality, with $\alpha>0$ chosen sufficiently small.

It remains to bound the fluctuation of $X$ under $P^\gamma_t$ from above. We divide into two cases: $\gamma=\infty$ or $\gamma\in (0,\infty)$.
\medskip

\noindent
{\em Case 1: $\gamma=\infty$}. By \eqref{eq:gibbsMeasRangeEst}, it suffices to show that
\begin{equation}\label{subdiffbd1}
\liminf_{t\to \infty} t^{-\frac{1}{3}-\epsilon} \!\!\!\!\!\!\!\!\!\! \inf_{X: \Vert X\Vert_t > t^{\frac{11}{24}+\epsilon}} \Big( \E^{Y}_0\big[\,\big|{\rm Range}_{s\in [0,t]}(Y_s+X_s)\big|\,\big] - \E^Y_0\big[\,\big|{\rm Range}_{s\in [0,t]}(Y_s)\big|\,\big]\Big) >0.
\end{equation}

Since $X$ and $Y$ are independent continuous time simple random walks, $Y+X$ is also a simple random walk, which allows us to write
\begin{equation}\label{3.21}
\begin{aligned}
\E^Y_0\Big[\,\big|{\rm Range}_{s\in [0,t]}(Y_s+X_s)\big|\,\Big] & = \E^Y_0\Big[\sup_{s\in [0,t]} (Y_s+X_s) - \inf_{s\in [0,t]} (Y_s+ X_s)  \Big] \\
& = \E^Y_0\Big[\sup_{s\in [0,t]} (Y_s+X_s) + \sup_{s\in [0,t]} (-Y_s-X_s)  \Big] \\
& = \E^Y_0\Big[\sup_{s\in [0,t]} (Y_s+X_s) + \sup_{s\in [0,t]} (Y_s-X_s)  \Big],
\end{aligned}
\end{equation}
where the last equality follows since $Y$ is symmetric.

Now observe that on the event $\{ \Vert X \Vert_t \ge t^{\frac{11}{24}+\epsilon}\}$, for
$$
\sigma(X) := {\rm arg\, max}_{s \in [0,t]} X_s \quad \text{and} \quad  \tau (X):= {\rm arg\, min}_{s \in [0,t]} X_s \in [0,t],
$$
where ties are broken by choosing the minimum value, one of the sets
$$
S:= \{s \in [0,t] \, : \,  X_{\sigma(X)} - X_s \ge t^{\frac{11}{24}+\epsilon}/2\} \quad \text{and} \quad
T:= \{r \in [0,t] \, : \,  X_s - X_{\tau(X)} \ge t^{\frac{11}{24}+\epsilon}/2\}
$$
has Lebesgue measure $\lambda(S)\geq t/2$ or $\lambda(T)\geq t/2.$ Without loss of generality, we can assume that $\lambda(S)\geq t/2$.
We then have
\begin{align}
& \E^Y_0\Big[\,\big|{\rm Range}_{s\in [0,t]}(Y_s+X_s)\big|\,\Big] - \E^Y_0\Big[\,\big|{\rm Range}_{s\in [0,t]}(Y_s)\big|\,\Big] \nonumber\\
= \ & \E^Y_0\Big[\sup_{s\in [0,t]} (Y_s+X_s) + \sup_{s\in [0,t]} (Y_s-X_s) - 2 \sup_{s\in [0,t]} (Y_s) \Big] \nonumber \\
\geq\ & \E^Y_0\Big[ \big(
\sup_{s\in [0,t]} (Y_s+X_s) + \sup_{s\in [0,t]} (Y_s-X_s) - 2 \sup_{s\in [0,t]} (Y_s) \big)
\ind{1}_{\sigma( Y) \in S}
\ind{1}_{Y_{\sigma(Y)} - Y_{\sigma(X)} \le t^{\frac{11}{24}}  } \Big] \nonumber \\
\geq\ & \E^Y_0\Big[ \big(
Y_{\sigma(X)}+X_{\sigma(X)} + Y_{\sigma(Y)}-X_{\sigma(Y)} - 2  Y_{\sigma(Y)} \big)
\ind{1}_{{\sigma(Y)} \in S}
\ind{1}_{Y_{{\sigma(Y)}} - Y_{\sigma(X)} \le t^{\frac{11}{24}} } \Big] \nonumber \\
\ge \ & (t^{\frac{11}{24}+\epsilon}/2 - t^{\frac{11}{24}}) \,
\P^Y_0 \big(   Y_{\sigma(Y)} - Y_{\sigma(X)} \le t^{\frac{11}{24}}, \, {\sigma(Y)} \in S \big), \label{subdiffbd2}
\end{align}
where the first inequality uses that the difference of the $\sup$'s in the expectation is non-negative.

It remains to lower bound the probability
\begin{align} \label{eq:condProbDec}
\P^Y_0 \big(   Y_{\sigma(Y)} - Y_{\sigma(X)} \le t^{\frac{11}{24}}, \,  {\sigma(Y)} \in S \big)
=\!\!  \int_S\!\! \P^Y_0 \!\big(   Y_{\sigma(Y)} - Y_{\sigma(X)} \le t^{\frac{11}{24}}\,
 \vert \, {\sigma(Y)}  = r \!\big) \P_0^Y \!( {\sigma(Y)} \in {\rm d}r \big).
\end{align}

Note that $\P^Y_0(\sigma(Y)\in {\rm d}r)$ is absolutely continuous with respect to  the Lebesgue measure $\lambda ({\rm d}r)$ with density
\begin{equation}\label{density}
\lim_{\delta\downarrow 0} \delta^{-1} \P^Y_0(\sigma(Y) \in [r, r+\delta]) = \rho\, \P^Y_0(Y_s \leq 0 \,\forall\, s\in [0, t-r]) \sum_{z<0}p_Y(z)\P^Y_z(Y_s<0 \,\forall\, s\in [0, r]),
\end{equation}
where $\rho$ is the jump rate of $Y$, $p_Y$ its jump kernel, $\sigma(Y)$ is the first time when $Y$ reaches its global maximum in the time interval $[0,t]$, and we used the observation that given $\sigma(Y)=r$ and the size of the jump at time $r$, $(Y_s)_{0\leq s\leq r}$ and $(Y_s)_{r\leq s\leq t}$ are two independent random walks. By \cite[Theorem 5.1.7]{LL10}, if $\tau_{[0,\infty)}$ denotes the first hitting time of $[0,\infty)$, then
there exist $C_1, C_2>0$ such that for all $z < 0$ and $s>|z|^2$,
\begin{equation} \label{eq:ruinAsymptotics}
C_1 \frac{\vert z \vert}{\sqrt s} \le
\P_z^Y \big(\tau_{[0,\infty)} \ge s \big) \le C_2 \frac{\vert z\vert }{\sqrt s}.
\end{equation}
Substituting the lower bound into \eqref{density}, we find that for any $\delta >0$, there exists a constant $c(\delta) >0$ such for all $t > 0$, we have
 \begin{equation} \label{eq:uniformBd}
 \P^Y_0(\sigma(Y) \in {\rm d}r) \ge \frac{c(\delta) \lambda({\rm d}r)}{t} \qquad \mbox{ on } [\delta t, (1-\delta) t].
 \end{equation}
Since $\lambda(S) \ge t/2$ by assumption, to lower bound the probability in \eqref{eq:condProbDec}, it only remains to lower bound
\begin{equation} \label{eq:besselLBprob}
\P^Y_0 \big(   Y_{\sigma(Y)} - Y_{\sigma(X)} \le t^{\frac{11}{24}}\,
 \vert \, {\sigma(Y)} = r \big)
 \end{equation}
uniformly in $r \in [\delta t, (1-\delta) t]$ with $|r-\sigma(X)|\geq \delta t$, and in $t>0$, for any $\delta<1/8$.

Note that conditioned on $\sigma(Y)=r$ and $Y_{\sigma(Y)^-}-Y_{\sigma(Y)}=z<0$, $(Y_{\sigma(Y)-s}-Y_{\sigma(Y)})_{s\in (0, r]}$ and $(Y_{\sigma(Y)+s}-Y_{\sigma(Y)})_{s\in [0, t-r]}$ are two independent conditioned random walks, starting respectively at $z$ and 0,
and conditioned respectively to not visit $[0,\infty)$ and $[1,\infty)$. Note that such conditioned random walks are comparable to a Bessel-3 process, although we will only use random walk estimates. We will only consider the case $s:=\sigma(X)-r>0$, the case $s<0$ is entirely analogous. We then get for \eqref{eq:besselLBprob} the lower bound
\begin{equation} \label{eq:LB}
\frac{\P^Y_0(  Y_s \ge -t^{\frac{11}{24}}, \, \tau_{[1,\infty)} \ge t-r )}
{\P_0^Y(\tau_{[1,\infty)} \ge t-r )}.
\end{equation}
By the Markov property, the numerator in \eqref{eq:LB} equals
\begin{align}
& \sum_{x=-\floor{t^{\frac{11}{24}}}}^{0}
\P^Y_0(Y_s=x, \tau_{[1,\infty)}\ge s)  \P_x^Y \big(\tau_{[1,\infty)} \ge t-r-s \big)\nonumber \\
\geq\ & \sum_{x=-\floor{t^{\frac{11}{24}}}}^{0}\sum_{\sqrt{s} \leq y, z\leq 2\sqrt{s}}
\P^Y_0(Y_{s/3}=y, Y_{2s/3}=z, Y_s=x, \tau_{[1,\infty)}\ge s) \P_x^Y \big(\tau_{[1,\infty)} \ge t-r-s \big) \nonumber\\
\geq\ & \sum_{x=-\floor{t^{\frac{11}{24}}}}^{0}\sum_{\sqrt{s} \leq y, z\leq 2\sqrt{s}}
\P^Y_0(Y_{s/3}=y, \tau_{[1,\infty)}>s/3) \P^Y_y(Y_{s/3}=z, \tau_{[1,\infty)}>s/3)  \nonumber\\
& \hspace{5cm} \times \P^Y_z(Y_{s/3}=x, \tau_{[1,\infty)}>s/3) \P_x^Y \big(\tau_{[1,\infty)} \ge t-r-s \big) \nonumber\\
\geq\ & C \sum_{x=-\floor{t^{\frac{11}{24}}}}^{0}\sum_{\sqrt{s} \leq y, z\leq 2\sqrt{s}}
\P^Y_0(Y_{s/3}=y, \tau_{[1,\infty)}>s/3)\cdot \frac{1}{\sqrt{s/3}}\cdot \P^Y_x(Y_{s/3}=z, \tau_{[1,\infty)}>s/3)\cdot \frac{|x-1|}{\sqrt{t-r-s}}  \nonumber \\
\geq\ & \frac{C}{t} \sum_{x=-\floor{t^{\frac{11}{24}}}}^{0}|x|\, \P^Y_0\big(Y_{s/3} \in [\sqrt{s}, 2\sqrt{s}], \tau_{[1,\infty)}>s/3\big)
\,\P^Y_x\big(Y_{s/3} \in [\sqrt{s}, 2\sqrt{s}], \tau_{[1,\infty)}>s/3\big) \nonumber\\
\geq\ & \frac{C}{t} \sum_{x=-\floor{t^{\frac{11}{24}}}}^{0}|x|\, \P^Y_0\big(\tau_{[1,\infty)}>s/3\big) \,\P^Y_x\big( \tau_{[1,\infty)}>s/3\big) \nonumber \\
\ge\ & \frac{C}{t^2} \sum_{x=-\floor{t^{\frac{11}{24}}}}^{0} |x|^2  \ge\  C t^{-2} t^{3\cdot \frac{11}{24}} = C t^{-\frac58}, \label{eq:besselBd}
\end{align}
where in the third inequality we applied the local limit theorem and \eqref{eq:ruinAsymptotics}, in the fourth inequality we used $s \geq \delta t$, and in the fifth inequality we used the fact that conditioned on $\{\tau_{[1,\infty)}>s/3\}$, $(Y_{us}/\sqrt{s})_{0\leq u\leq 1/3}$ converges in distribution to a Brownian meander if $Y_0\ll \sqrt{s}$ as $s\to\infty$ (cf.~\cite{B76}).

Since $t-r\geq \delta t$, again by \eqref{eq:ruinAsymptotics}, we find that
\begin{equation*}
 \frac{\P^Y_0(  Y_s \ge -t^{\frac{11}{24}}, \, \tau_{[1,\infty)} \ge t )}
{\P_0^Y(\tau_{[1,\infty)} \ge t-r )}\ge
C t^{-\frac58} t^{\frac12} = Ct^{-1/8}.
\end{equation*}
Plugging this into \eqref{eq:condProbDec} (recall \eqref{eq:uniformBd}) and the resulting inequality into \eqref{subdiffbd2},
we find that \eqref{subdiffbd1} holds, which concludes the proof for the case $\gamma=\infty$ and $X,Y$ are simple random walks.
\medskip

\noindent
{\em Case 2: $\gamma\in (0,\infty)$.} We first use \eqref{eq:gibbsMeasRangeEst} to upper bound
\begin{align}
& P^\gamma_{t} \big( \Vert X  \Vert_t \ge t^{\frac{11}{24}+\epsilon} \big) \nonumber\\
\le\ &
e^{c_1 t^{\frac{1}{3}}} \E^X_0\Big[\exp\Big\{-\nu \Big(\E^{Y,N}_0\big[\,\big|{\rm SoftRange}_{s\in [0,t]}(Y_s+X_s)\big|\,\big]- \E^Y_0\big[\,\big|{\rm Range}_{s\in [0,t]}(Y_s)\big|\,\big]\Big) \Big\}
\ind{1}_{\Vert X  \Vert_t \ge t^{\frac{11}{24}+\epsilon}}  \Big]  \nonumber\\
=\ &
e^{c_1 t^{\frac{1}{3}}} \E^X_0\Big[\exp\Big\{-  \nu \Big(\underbrace{
\E^Y_0\big[\,\big|{\rm Range}_{s\in [0,t]}(Y_s+X_s)\big|\,\big]- \E^Y_0\big[\,\big|{\rm Range}_{s\in [0,t]}(Y_s)\big|\,\big]}_{\stackrel{\eqref{subdiffbd1}}{\geq} ct^{\frac13+\epsilon}} \Big)\Big\}  \nonumber\\
&\qquad \qquad \qquad \times \exp\Big\{ \nu \E_0^Y[ F_t^{\gamma} (X+Y)] \Big\}
\ind{1}_{\Vert X  \Vert_t \ge t^{\frac{11}{24}+\epsilon}}  \Big], \label{eq:largeFluct}
\end{align}
where we applied \eqref{RsoR} to $Y+X$ in the last equality.

It is clear from the above bound that
\begin{equation}\label{3.31}
P^\gamma_{t} \big( \Vert X  \Vert_t \ge t^{\frac{11}{24}+\epsilon}, \E_0^Y[ F_t^{\gamma} (X+Y)]\leq t^{\frac13+\frac{\epsilon}{2}}\big) \asto{t}{0}.
\end{equation}
On the other hand, by the same calculations as in \eqref{3.19}, we have
\begin{equation}\label{3.32}
\begin{split}
P^\gamma_t\Big(\E^Y_0[F_t^{\gamma} (X+Y)]> t^{\frac13+\frac{\epsilon}{2}}\Big)
\leq \ & e^{c_1 t^{\frac{1}{3}}} e^{\nu \E_0^Y[ F_t^{\gamma} (Y)]} \, \P^X_0\big(\E^Y_0[F_t^{\gamma} (X+Y)]> t^{\frac13+\frac{\epsilon}{2}}\big) \\
\leq \ & e^{c_1 t^{\frac{1}{3}}+C\nu \ln t} e^{-\frac{c(\gamma)t^{1/3+ \epsilon/2}}{1\vee \ln t}}
\E^X_0\Big[e^{\frac{c(\gamma)}{1\vee\ln t} \E^Y_0[F_t^{\gamma} (X+Y)]}\Big] \\
\leq \ & e^{c_1 t^{\frac{1}{3}}+C\nu \ln t -\frac{c(\gamma)t^{1/3+ \epsilon/2}}{1\vee \ln t}}
\E^X_0\E^Y_0\Big[e^{\frac{c(\gamma)}{1\vee\ln t} F_t^{\gamma} (X+Y)}\Big] \\
\leq \ & C(\gamma) e^{c_1 t^{\frac{1}{3}}+C\nu \ln t -\frac{c(\gamma)t^{1/3+ \epsilon/2}}{1\vee \ln t}} \asto{t}{0},
\end{split}
\end{equation}
where we have applied Jensen's inequality and Proposition \ref{prop:expMomLocal}. Combined with \eqref{3.31}, this concludes the proof for the case $\gamma \in (0,\infty)$.
\end{proof}
\medskip

\begin{proof}[{Proof of Theorem \ref{thm:main} for general $X$ and $Y$}]
When $X$ and $Y$ are non-simple random walks, identities <such as \eqref{3.21} fails because the range of the walk is no longer the interval bounded between the walk's infimum and supremum. Theorem \ref{thm:range} allows us to salvage the argument.

The lower bound \eqref{3.19} on the fluctuation of $X$ under $P^\gamma_t$ remains valid, using \eqref{eq:gibbsMeasRangeEst} for $Y$ as in (\ref{grwy}) and (\ref{expmom}).

For the upper bound on the fluctuation of $X$ under $P^\gamma_t$, $\gamma\in (0,\infty]$, note that by \eqref{eq:gibbsMeasRangeEst},
\begin{align}
& P^\gamma_{t} \big( \Vert X  \Vert_t \ge t^{\frac{11}{24}+\epsilon} \big) \nonumber\\
\le\ &
e^{c_1 t^{\frac{1}{3}}} \E^X_0\Big[\exp\Big\{-\nu \Big(\E^{Y,N}_0\big[\,\big|{\rm SoftRange}_{s\in [0,t]}(Y_s+X_s)\big|\,\big]- \E^Y_0\big[\,\big|{\rm Range}_{s\in [0,t]}(Y_s)\big|\,\big]\Big) \Big\}
\ind{1}_{\Vert X  \Vert_t \ge t^{\frac{11}{24}+\epsilon}}  \Big]  \nonumber\\
=\ &
e^{c_1 t^{\frac{1}{3}}} \E^X_0\Big[\exp\Big\{-  \nu \Big(
\E^Y_0\big[\,\big|{\rm Range}_{s\in [0,t]}(Y_s+X_s)\big|\,\big]- \E^Y_0\big[\,\big|{\rm Range}_{s\in [0,t]}(Y_s)\big|\,\big] \Big)\Big\}  \nonumber\\
&\qquad \qquad \qquad \times \exp\Big\{ \nu \E_0^Y[ F_t^{\gamma} (X+Y)] \Big\}
\ind{1}_{\Vert X  \Vert_t \ge t^{\frac{11}{24}+\epsilon}}  \Big] \nonumber \\
=\ &
e^{c_1 t^{\frac{1}{3}}} \E^X_0\Big[\exp\Big\{-  \nu \Big(
\E^Y_0\big[\,\sup_{s\in [0,t]}(Y_s+X_s) -\inf_{s\in [0,t]}(Y_s+X_s)\,\big]- \E^Y_0\big[\,\sup_{s\in [0,t]}Y_s -\inf_{s\in [0,t]}Y_s\,\big] \Big)\Big\}  \nonumber\\
&\qquad \qquad \qquad \times \exp\Big\{ \nu \E^Y_0[G_t(Y+X)] + \nu \E_0^Y[ F_t^{\gamma} (X+Y)] \Big\}
\ind{1}_{\Vert X  \Vert_t \ge t^{\frac{11}{24}+\epsilon}}  \Big] \nonumber \\
=\ &
e^{c_1 t^{\frac{1}{3}}} \E^X_0\Big[\exp\Big\{-  \nu \Big(
\E^Y_0\big[\,\sup_{s\in [0,t]}(Y_s+X_s) +\sup_{s\in [0,t]}(Y_s-X_s) - 2 \sup_{s\in [0,t]}Y_s\,\big] \Big)\Big\}  \nonumber\\
&\qquad \qquad \qquad \times \exp\Big\{ \nu \E^Y_0[G_t(Y+X)] + \nu \E_0^Y[ F_t^{\gamma} (X+Y)] \Big\}
\ind{1}_{\Vert X  \Vert_t \ge t^{\frac{11}{24}+\epsilon}}  \Big] \label{5.18}
\end{align}
where we recall from \eqref{GtX} that
$G_t(X) := \big(\sup_{0\leq s\leq t} X_s - \inf_{0\leq s\leq t} X_s\big) - |{\rm Range}_{s\in [0,t]}(X_s)|$,
and we used the symmetry of $Y$ in the last equality. Note that when $\gamma=\infty$, $F^\gamma_t(Y+X)=0$.

The proof of \eqref{subdiffbd2} does not require $X$ and $Y$ to be simple random walks, and in particular, it implies that
$$
\liminf_{t\to \infty} t^{-\frac{1}{3}-\epsilon} \!\!\!\!\!\!\!\!\!\! \inf_{X: \Vert X\Vert_t > t^{\frac{11}{24}+\epsilon}} \E^Y_0\big[\,\sup_{s\in [0,t]}(Y_s+X_s) +\sup_{s\in [0,t]}(Y_s-X_s) - 2 \sup_{s\in [0,t]}Y_s\,\big] >0.
$$
Therefore it follows from \eqref{5.18} that
\begin{equation}\label{5.19}
P^\gamma_{t} \big( \Vert X  \Vert_t \ge t^{\frac{11}{24}+\epsilon}, \, \E_0^Y[ F_t^{\gamma} (X+Y)]\leq t^{\frac13+\frac{\epsilon}{2}}, \,
\E_0^Y[ G_t(Y+X)]\leq t^{\frac13+\frac{\epsilon}{2}} \big) \asto{t}{0}.
\end{equation}
The argument for $P^\gamma_t(\E_0^Y[ F_t^{\gamma} (X+Y)]> t^{\frac13+\frac{\epsilon}{2}})\to 0$ in \eqref{3.32} is still valid, while the same argument as in \eqref{3.32} with $F^\gamma_t(Y+X)$ replaced by $G_t(Y+X)$, together with Theorem \ref{thm:range}, shows that we also have
$P^\gamma_t(\E_0^Y[ G_t (Y+X)]> t^{\frac13+\frac{\epsilon}{2}})\to 0$ as $t\to\infty$. This completes the proof.
\end{proof}

\bigskip

\noindent{\textbf{Acknowledgement.}} The authors would like to thank
the Columbia University Mathematics Department, the Forschungsinstitut
f\"ur Mathematik at ETH Z\"urich, Indian Statistical Institute
Bangalore, and the National University of Singapore for hospitality
and financial support. R.S.\ is supported by NUS grant
R-146-000-185-112. S.A is supported by CPDA grant from the Indian
Statistical Institute.


\begin{thebibliography}{DGRS12}


\bibitem[B76]{B76}
E. Bolthausen.
On a functional central limit theorem for random walks conditioned to stay positive.
{\em Ann. Prob.} 4(3):480–-485, 1976.

\bibitem[DGRS12]{DrGaRaSu-10}
A. Drewitz, J. G\"artner, A.F. Ram{\'{\i}}rez, and R. Sun.
Survival probability of a random walk among a Poisson system of
  moving traps.
{\em Probability in Complex Physical Systems --- In honour of Erwin
  Bolthausen and J\"urgen G\"artner, Springer Proceedings in Mathematics.},
  11:119--158, 2012.

\bibitem[DPRZ00]{DePeRoZe-00}
A. Dembo, Y. Peres, J. Rosen, and O. Zeitouni.
Thin points for Brownian motion.
{\em Ann. Inst. H. Poincar\'e Probab. Statist.}, 36(6):749--774,
  2000.

\bibitem[DV75]{DoVa-75}
M.~D. Donsker and S.~R.~S. Varadhan.
Asymptotics for the {W}iener sausage.
{\em Comm. Pure Appl. Math.}, 28(4):525--565, 1975.

\bibitem[DV79]{DoVa-79}
M.~D. Donsker and S.~R.~S. Varadhan.
On the number of distinct sites visited by a random walk.
{\em Comm. Pure Appl. Math.}, 32(6):721--747, 1979.

\bibitem[\rm LL10]{LL10}
G.F. Lawer and V. Limic.
{\em Random walk: a modern introduction}, volume 123 of {\em Cambridge Studies in Advanced Mathematics}.
Cambridge University Press, Cambridge, 2010.

\bibitem[MOBC03]{MOBC03}
M.\ Moreau, G.\ Oshanin, O.\ B\'enichou and M.\ Coppey.
Pascal principle for diffusion-controlled trapping reactions.
{\em Phys.\ Rev.\ E} 67, 045104(R), 2003.

\bibitem[MOBC04]{MOBC04}
M.\ Moreau, G.\ Oshanin, O.\ B\'enichou and M.\ Coppey.
Lattice theory of trapping reactions with mobile species.
{\em Phys.\ Rev.\ E} 69, 046101, 2004.

\bibitem[PSSS13]{PSSS13}
Y.~Peres, A.~Sinclair, P.~Sousi, and A.~Stauffer.
Mobile geometric graphs: detection, coverage and percolation.
{\em Probab. Theory Related Fields} 156, 273--305, 2013.

\bibitem[RD69]{RD69}
B.C. Rennie and A.J. Dobson.
On Stirling numbers of the second kind.
{\em Journal of Combinatorial Theory} 7(2):116--121, 1969.

\bibitem[S90]{S90}
U. Schmock.
Convergence of the normalized one-dimensional Wiener sausage path measures to a mixture of Brownian taboo processes.
{\em Stochastics Stochastics Rep.} 29(2):171--183, 1990.

\bibitem[S03]{S03}
S. Sethuraman.
Conditional survival distributions of Brownian trajectories in a one dimensional Poissonian environment.
{\em Stochastic Process. Appl.} 103(2):169--209, 2003.

\bibitem[Szn98]{Sz-98}
A.-S. Sznitman.
{\em Brownian motion, obstacles and random media}.
Springer Monographs in Mathematics. Springer-Verlag, Berlin, 1998.

\end{thebibliography}
\end{document}